\documentclass[11pt,a4paper]{article}
\usepackage[T1]{fontenc}
\usepackage[utf8]{inputenc}
\usepackage[english]{babel}
\usepackage[a4paper,margin=1.4in]{geometry}
\usepackage{setspace}
\usepackage{parskip} 
\usepackage{amsmath,amssymb,amsfonts,amsthm,mathtools}
\usepackage{bm}
\usepackage{mathrsfs}
\usepackage{graphicx}
\usepackage{float}
\usepackage{subcaption}
\usepackage{caption}
\usepackage{tikz}
\usepackage{pgfplots}
\pgfplotsset{compat=1.18}
\usepackage{array}
\usepackage{booktabs}
\usepackage{multirow}
\usepackage{tabularx}
\usepackage{algorithm}
\usepackage{algpseudocode}
\usepackage{listings}
\usepackage{xcolor}
\usepackage[numbers,sort&compress]{natbib}
\usepackage[colorlinks=true,linkcolor=blue,citecolor=blue,urlcolor=blue]{hyperref}
\usepackage[nameinlink,capitalise,noabbrev]{cleveref}
\usepackage{enumitem}
\usepackage{url}
\usepackage{doi}
\usepackage{authblk}
\usepackage{orcidlink}
\newtheorem{theorem}{Theorem}[section]
\newtheorem{lemma}[theorem]{Lemma}

\theoremstyle{definition}
\newtheorem{definition}[theorem]{Definition}

\title{\textbf{Distributionally Robust Complex Chance-Constrained Optimization}}

\author[1]{Raneem Madani}
\author[1]{Abdel Lisser}
\author[1]{Zeno Toffano}
\affil[1]{Universit\'e Paris-Saclay, CNRS, CentraleSup\'elec, Laboratoire des Signaux et Syst\`emes (L2S), 91190, Gif-sur-Yvette, France\\
\texttt{\{raneem.madani,abdel.lisser,zeno.toffano\}@centralesupelec.fr}}

\date{}
\begin{document}

\maketitle

\begin{abstract}
This paper introduces a framework for Chance-Constrained Optimization with Complex Variables, addressing complex linear programming for both individual and joint probabilistic constraints in the complex domain. We first analyze the 3CP model in the density-based setting under the assumption that the random parameters follow a Complex Elliptically Symmetric distribution. The framework is then extended to distributionally robust settings, which include a moment-based model where the moments are known or bounded; a support-based model, where the ambiguity set contains distributions supported on norm-bounded uncertainty sets; and a data-driven model where moments are estimated empirically. The individual constraints are transformed into a convex deterministic second-order cone problem. We employ copula theory to the joint probability constraints and derive both upper and lower approximations. Finally, we demonstrate the proposed framework on the minimum-variance distortionless response beamforming problem in signal processing. We further evaluate empirical out-of-sample rates and show that the observed behavior closely matches the prescribed probabilistic guarantees.
\end{abstract}

\noindent\textbf{Keywords:} Chance-constrained problem, elliptical distribution, distributionally robust optimization, data-driven optimization, complex variables.
\section{Introduction} 
\label{sec:1}
Chance-constrained programming (CCP), which was introduced by Charnes and Cooper \cite{charnes1959chance}, where constraints are required to hold with prescribed confidence levels, thus balancing feasibility and risk~\cite{CHENG2012325,LIU2016687}. When the probability distribution of the uncertainty is only partially known, Distributionally Robust Optimization (DRO) addresses this issue by optimizing against the worst-case distribution within an ambiguity set, yielding solutions that are robust to model misspecification~\cite{calafiore2006distributionally,OJMO_2022__3__A4_0}. This has produced a substantial literature on distributionally robust chance constraints and tractable reformulations, predominantly in the real-valued setting~\cite{doi:10.1137/130915315,doi:10.1287/moor.2021.1233,NGUYEN2023100285,doi:10.1137/24M1660711}. In complex-valued problems, Complex Elliptically Symmetric (CES) distributions form a broad and flexible class that extends the complex Gaussian family while accommodating non-Gaussian behaviors such as heavier tails~\cite{6263313}. Because they include important models such as the complex normal, complex $t$, generalized Gaussian, and compound-Gaussian distributions, CES distributions are widely used in signal processing, radar, communications, and imaging applications~\cite{5313943,1091566}.

A wide range of real-world problems are naturally formulated with complex-valued decision variables, where both magnitude and phase are essential~\cite{10452289}. Such formulations arise, for example, in MIMO systems~\cite{10.1007/978-3-030-27192-3_1,10.1007/978-3-031-39764-6_15} and in quantum information~\cite{Lauro:25}. This has motivated the development of optimization methods directly in the complex domain, including deterministic linear~\cite{LEVINSON196644}, nonlinear~\cite{ABRAMS1972619,FERRERO1992399}, and quadratic formulations~\cite{doi:10.1137/04061341X}, as well as more recent approaches based on CR-calculus~\cite{kreutz2009complex,doi:10.1137/110832124}. More recently, we introduced the Complex Chance-Constrained Problem (3CP) in~\cite{10.1007/978-3-032-13589-6_21} as a general complex-valued extension of CCP, focusing on linear problems under complex Gaussian uncertainty and yielding deterministic convex reformulations together with second-order cone program (SOCP) approximations for both individual and joint chance constraints. In \cite{madani2026wasserstein}, we study the data-driven 3CP distributionally robust setting under Wasserstein ambiguity sets centered at the empirical distribution. In that work, tractable approximations are derived through worst-case CVaR reformulations. In particular, individual chance constraints lead to a convex and conservative program, while joint chance constraints result in a biconvex approximation handled by a convex sequential method and a spatial branch-and-bound procedure providing certified bounds.

Building on \cite{10.1007/978-3-032-13589-6_21}, in this paper, we extend 3CP to encompass a variety of uncertainty settings. We first consider the density-based case where the random parameters follow a CES distribution. We then address distributionally robust scenarios, beginning with (i) moments-based cases where the first and second moments are known, or they are bounded, but the distribution is unknown, and later (ii) only the mean and norm bounds of the data are available. A symmetric special case of the bounded model is also examined. Subsequently, (iii) we consider data-driven approaches where the underlying distribution is unknown, but samples are available. In this case, the mean, covariance, and pseudo-covariance are estimated empirically, and a moment-robust reformulation is developed to explicitly account for estimation errors in the chance constraints. Across all these settings, the stochastic constraints are transformed into deterministic convex formulations based on the available statistical information. For individual chance constraints, the resulting problems are shown to be convex SOCP. For joint chance constraints, copula theory is used to decouple dependencies, and both upper and lower SOCP approximations are derived using piecewise linearization and Taylor expansion. Finally, we validate the proposed formulations through extensive numerical experiments and an application to MVDR beamforming under complex Gaussian uncertainty. Additionally, our experiments across a range of complex uncertainty models show that empirical violation frequencies remain below the prescribed risk level; for the joint chance-constraint formulation, the upper-lower SOCP gap decreases as the number of tangent points increases, demonstrating the tightness of the proposed approximation.

The flow of the paper is organized as follows. Section~\ref{sec: 2} presents the preliminaries and introduces the main mathematical tools and notation used throughout the paper. Section~\ref{sec: 3} formulates the linear 3CP. Section~\ref{sec: 4} develops the 3CP framework under the assumption of CES-distributed parameters. Section~\ref{sec: 5} extends the formulation to the distributionally robust setting. In Section~\ref{sec: 6}, we derive SOCP upper and lower bound approximations for the joint chance constraints. Section~\ref{sec: 7} presents numerical results and a real-world application, followed by concluding remarks in Section~\ref{sec: 8}.
\section{Preliminaries} \label{sec: 2}
The space of complex vectors of dimension $n$ is denoted by $\mathbb{C}^n$. Any complex vector $z\in\mathbb{C}^n$ can be expressed in terms of its real and imaginary parts as $z = x+iy$,
where $x=\mathrm{Re}(z)\in\mathbb{R}^n$ and $y=\mathrm{Im}(z)\in\mathbb{R}^n$ \cite{ahlfors1979complex}. The complex conjugate of $z$ is defined by $\bar z = x-iy$, while its transpose and Hermitian transpose are denoted by $z^T$ and $z^H$, respectively. The Euclidean norm of $z$ is given by $\|z\| = \sqrt{z^H z} = \sqrt{\|x\|^2+\|y\|^2}$. 

A complex random vector can be viewed as a vector of complex-valued random variables, or equivalently, as a pair of real random vectors corresponding to its real and imaginary parts. More precisely, a complex random vector $z=(z_1,\dots,z_n)^T$ defined on a probability space $(\Omega,\mathcal{F},P)$ is a measurable mapping $z:\Omega \to \mathbb{C}^n$ such that $[\mathrm{Re}(z_1),\mathrm{Im}(z_1),\dots,\mathrm{Re}(z_n),\mathrm{Im}(z_n)]^T$ is a real random vector on $(\Omega,\mathcal{F},P)$. A random vector is characterized by its mean, covariance matrix, and pseudo-covariance matrix, defined as follows \cite{1179767}.
\begin{definition}\label{Def: 1}
Let $z=x+iy$ be a complex random vector. Its mean, covariance matrix, and pseudo-covariance matrix are defined as
\begin{align}
    &\mu_z = \mathbb{E}[z] = \mathbb{E}[x] + i\,\mathbb{E}[y].\\
    &\Gamma_z = \mathrm{Cov}(z,z) = \mathbb{E}\left[(z-\mathbb{E}[z])(z-\mathbb{E}[z])^H\right] = \Gamma_x + \Gamma_y + i(\Gamma_{yx}-\Gamma_{xy}),\\
    &J_z = \mathrm{Cov}(z,\bar z) = \mathbb{E}\left[(z-\mathbb{E}[z])(z-\mathbb{E}[z])^T\right] = \Gamma_x - \Gamma_y + i(\Gamma_{yx}+\Gamma_{xy}).
\end{align}
where the $(i,j)$th entry of $\Gamma_z$ represents the covariance between the $i$th and $j$th components of $z$.
\end{definition}
The covariance matrix $\Gamma_z$ is Hermitian positive semidefinite, that is, $\Gamma_z^H=\Gamma_z$, whereas the pseudo-covariance matrix $J_z$ is symmetric, namely $J_z^T=J_z$. In addition, the covariance operator satisfies the following properties: $\mathrm{Cov}(z,w)=\overline{\mathrm{Cov}(w,z)},$ and for any $\alpha\in\mathbb{C}$, $\mathrm{Cov}(\alpha z,w)=\alpha\,\mathrm{Cov}(z,w), \quad\mathrm{Cov}(z,\alpha w)=\bar{\alpha}\,\mathrm{Cov}(z,w).$ A copula theory combines several univariate distributions into a valid multivariate distribution, thereby allowing the dependence structure to be modeled independently from the marginals.
\begin{definition}[Copula \cite{nelsen2006introduction}]
A copula of dimension $m$ $(m\ge2)$ is an $m$-dimensional cumulative distribution function $\mathcal{C}:[0,1]^m \to [0,1]$ whose univariate marginals are all uniform on $[0,1]$.
\end{definition}

In this paper, we focus on a specific family of copulas, namely the Gumbel-Hougaard copula, denoted by $\mathcal{C}_\theta$.
\begin{definition}[Gumbel-Hougaard Copula \cite{nelsen2006introduction,cheng2015chance}]\label{Gumbel} 
For $u=(u_1,\dots,u_m)\in[0,1]^m$ and $\theta\ge 1$, the Gumbel-Hougaard copula is defined as
\begin{align*}
    \mathcal{C}_\theta(u)
    = \exp\left\{
    -\left[\sum_{i=1}^m(-\ln u_i)^\theta\right]^{1/\theta}
    \right\}.
\end{align*}
\end{definition}

When $\theta=1$, the Gumbel-Hougaard copula reduces to the independence copula, $\mathcal{C}_1(u_1,\dots,u_m)=\prod_{i=1}^m u_i,$ which means that the joint cumulative probability equals the product of the marginal probabilities. Hence, $\theta=1$ corresponds to statistical independence among the probabilistic constraints, whereas values $\theta>1$ model increasing positive dependence. The connection between multivariate distributions and copulas is provided by Sklar’s theorem, which guarantees the existence of a copula representation for any multivariate distribution, and its uniqueness when the marginals are continuous.
\begin{theorem}[Sklar’s theorem \cite{nelsen2006introduction,doi:10.1137/24M1660711}]\label{sklar}
Let $\hat{\Phi}$ be an $m$-dimensional distribution function with one-dimensional marginals $\hat{\Phi}_1,\hat{\Phi}_2,\dots,\hat{\Phi}_m$. Then there exists a copula $\mathcal{C}$ such that
\[
\hat{\Phi}(u_1,u_2,\dots,u_m)
=
\mathcal{C}\bigl(\hat{\Phi}_1(u_1),\hat{\Phi}_2(u_2),\dots,\hat{\Phi}_m(u_m)\bigr),
\qquad \forall\, u_1,u_2,\dots,u_m\in\mathbb{R}.
\]
\end{theorem}
\section{Problem Formulation} \label{sec: 3} 
The linear complex chance constraint program is given by
\begin{equation}
\begin{aligned}
        &\min \mathrm{Re}(c^Hz)\\
        &s.t.~ \mathbb{P}[\mathrm{Re}(Az)\leq b]\ge p,\\
\end{aligned}\label{3CP}
\end{equation}
where $c\in\mathbb{C}^n, A\in\mathbb{C}^{m\times n}, b\in\mathbb{R}^m$ are the random vectors, $p$ is the probability value. Define the random vector, for $i=1,\cdots,m$:
\begin{align}
    &d_i=[a_i ~b_i]\\
    &\mu_{d_i}=\mathbb{E}[d_i]=[\mu_{a_i}~ \mu_{b_i}]\\
    &\Gamma_{d_i}=var[d_i]=var([a_i~ b_i])\succeq0
\end{align}
With $A=[a_1, \cdots, a_m], a_i\in\mathbb{C}^n, b_i\in \mathbb{R}, i=1,\cdots,m$. We further define the scalar quantity, let $\tilde{z}=[z^T,~1]^T,\quad i=1,\cdots,m$:
\begin{align}
    &\varphi_i(z)=\mathrm{Re}(d_i\Tilde{z})\\
    &\mu_{\varphi_i(z)}=\mathbb{E}[\varphi_i(z)],\\
    &\sigma_{\varphi_i(z)}^2=var[\varphi_i(z)],
\end{align}
\begin{theorem}\label{theorem 1}
If the random vector $(\varphi_1(z), \dots, \varphi_m(z))^{T}$, 
whose joint distribution is governed by the Gumbel--Hougaard copula 
$\mathcal{C}_\theta$ with some $\theta \ge 1$, 
then, for any $p\in(0,1)$, the joint probability constraint 
$\mathbb{P}[\varphi_i(z)\leq 0,~ i=1,\cdots,m]\ge p$ 
is equivalent to:
\begin{align}
   \sum_i y_i=1; \quad y_i>0, \quad \mathbb{P}[\varphi_i(z)\leq 0]\ge p^{y_i^{\frac{1}{\theta}}}, 
    \qquad i=1,\cdots,m.
\end{align}
\end{theorem}
\begin{proof}
Let 
\[
\tilde{\varphi}_i(z)=
\frac{\varphi_i(z)-\mu_{\varphi_i(z)}}{\sigma_{\varphi_i(z)}},\quad i=1,\cdots,m
\]
be the standardized random variable, and let
\[
\xi_i(z)=
-\frac{\mu_{\varphi_i(z)}}{\sigma_{\varphi_i(z)}},\quad i=1,\cdots,m
\]
be the corresponding deterministic threshold. With this notation, the joint complex chance constraint can be equivalently rewritten as:
\begin{align}
    \mathbb{P}[\varphi_i(z)\le0,~ i=1,\cdots,m]
    =\mathbb{P}\left[\tilde{\varphi}_i(z)\le\xi_i(z),~ i=1,\cdots,m\right]
    \ge p.
\end{align}
Assume that there exist $y_i$, such that $\sum_i y_i=1$ and 
\[
\Phi(\xi_i(z))\ge p^{y_i^{1/\theta}}, 
\qquad  i=1,\cdots,m.
\]
where $\Phi(\cdot)$ denotes the marginal cumulative distribution function (CDF) of the random variable $\xi_i(z)$. From the definition of the Gumbel--Hougaard copula \eqref{Gumbel} 
and Sklar’s theorem \eqref{sklar}, we have
\begin{align*}
    &\mathbb{P}\big[\tilde{\varphi}_i(z) \le \xi_i(z),~ i=1,\cdots,m\big]
    = \mathcal{C}_\theta\big(\Phi_1(\xi_1(z)), \dots,\Phi_m(\xi_m(z))\big) \\
    &\ge \mathcal{C}_\theta\big(p^{y_1^{1/\theta}}, \dots, p^{y_m^{1/\theta}}\big) = \exp\!\left\{-\left[\sum_{i=1}^{m} 
        \big(-\ln p^{y_i^{1/\theta}}\big)^{\theta}\right]^{1/\theta}\right\} \\
    &= \exp\!\left\{-\left[\sum_{i=1}^{m} 
        (-y_i^{1/\theta}\ln p)^{\theta}\right]^{1/\theta}\right\} = \exp\!\left\{\ln p
        \left[\sum_{i=1}^{m} y_i\right]^{1/\theta}\right\}= p.
\end{align*}
\end{proof}
For simplicity, we define \( p_i = p^{y_i^{\frac{1}{\theta}}} \) as a function of the weighting variable \( y_i\) in the joint 3CP formulation. In contrast, in the individual 3CP formulation, each \( p_i \) is treated as a fixed constant. For illustration, we study the simple single-constraint case: \(\mathbb{P}\!\left[\mathrm{Re}(a_i^H z - b_i) \le 0\right] \ge p_i\). In the following two lemmas, we introduce the distributions of the complex affine functions in problem \eqref{3CP}.
\begin{lemma} 
Let $c$ be a random vector with mean $\mu_{c}$, covariance $\Gamma_{c}$, and pseudo-covariance $J_c$, then $\mathrm{Re}(c^Hz)$ has a real mean and variance, which are given by:
\begin{align}
    \mu_{\mathrm{Re}(c^Hz)}= \mathrm{Re}(\mu_c^Hz),\quad\Gamma_{\mathrm{Re}(c^Hz)}=\frac{1}{2}(z^H\Gamma_cz+\mathrm{Re}(z^HJ_c\bar{z}))
\end{align}
Furthermore, the variance is a second-order function.
\end{lemma}
\begin{proof}
Using Definition \ref{Def: 1}, then the mean given by $\mathbb{E}[\mathrm{Re}(c^Hz)]=\mathrm{Re}(\mu_c^Hz)$. Let $z=x+iy$ and $c=c^r+ic^i$, then the variance given by:
\begin{align*}
    Var(\mathrm{Re}(c^Hz))&=Var(c^{rT}x+c^{iT}y)\\
    &=Var(c^{rT}x)+Var(c^{iT}y)+Cov(c^{rT}x,c^{iT}y)+Cov(c^{iT}y,c^{rT}x)\\
    &=x^T\Gamma_{c^r}x+x^T\Gamma_{c^rc^i}y + y^T\Gamma_{c^ic^r}x+y^T\Gamma_{c^i}y\\
    &=\frac{1}{2}\left(z^H\Gamma_cz + \mathrm{Re}(z^HJ_c\bar{z})\right)=w^TKw
\end{align*}
where $K = \begin{bmatrix}
     \Gamma_{c^r} & \Gamma_{c^rc^i}\\ \Gamma_{c^ic^r} &  \Gamma_{c^i}
\end{bmatrix}\succeq0$, and $w=\begin{bmatrix} x\\y \end{bmatrix}$ which is second-order function.
\end{proof}
In the following lemma, without loss of generality, we assume that the random variables $b, a$ are independent.
\begin{lemma}
    Consider the complex random column vector $a=(a_{1},\cdots, a_{n})\in \mathbb{C}^n$ with mean $\mu_{a}$, covariance matrix $\Gamma_{a}$ and pseudo-covariance matrix $J_a$; $b\in\mathbb{R}$ has mean $\mu_{b}$ and variance $\sigma_b^2$, then $\mathrm{Re}(az-b)$ has real mean and variance: 
\begin{align}
    &\mu_{\mathrm{Re}(az-b)} = \mathrm{Re}(\mu_{a}z)-\mu_{b},\\
    &\Gamma_{\mathrm{Re}(az-b)}=\frac{1}{2}(z^H\Gamma_{a}z+\mathrm{Re}(z^TJ_{a}z))+\sigma_{b}^2
\end{align}
Furthermore, the variance is a second-order function.
\end{lemma}
\begin{proof}
Using Definition \ref{Def: 1}, then the mean given by $\mathbb{E}[\mathrm{Re}(az)]=\mathrm{Re}(\mu_az)$. Let $z=x+iy$ and $a=a^r+ia^i$, then the variance given by:
\begin{align*}
    &Var(\mathrm{Re}(az)-b)=Var(a^{r}x-a^{i}y-b)\\
    &=Var(a^{r}x)+Var(a^{i}y)+Cov(a^{r}x,-a^{i}y)+Cov(-a^{i}y,a^{r}x)+Var(b)\\
    &=x^T\Gamma_{a^r}x-x^T\Gamma_{a^ra^i}y - y^T\Gamma_{a^ia^r}x+y^T\Gamma_{a^i}y+\sigma_b^2\\
    &=\frac{1}{2}\left(z^H\Gamma_az + \mathrm{Re}(z^TJ_az)\right)+\sigma_b^2=w^TKw
\end{align*}
where $K = \begin{bmatrix}
     \Gamma_{a^r} & \Gamma_{a^ra^i} & 0\\ \Gamma_{a^ia^r} &  \Gamma_{a^i}& 0\\ 0 &0 & \sigma^2_b
\end{bmatrix}\succeq0$, and $w=\begin{bmatrix} x\\-y\\1 \end{bmatrix}$ which is second-order function.
\end{proof}
\section{Density Based: Complex Elliptically Symmetric Distribution} \label{sec: 4}
In this section, we study the density-based setting, where the uncertainty distribution is assumed to be known and to belong to the family of complex elliptically symmetric (CES) distributions. Leveraging key properties of CES laws, we transform the resulting stochastic problem into a deterministic equivalent one.
\begin{definition}[Complex Elliptically Symmetric (CES) Distribution {\cite{1502990,6263313}}]
A random vector $d \in \mathbb{C}^{n+1}$ with mean $\mu_d$, covariance $\Gamma_d$, and pseudo-covariance $J_d$ is said to follow a  {complex elliptically symmetric (CES)} distribution if its characteristic function is given by
\begin{equation}
    \Xi_d(c) = \exp\!\big(i\,\mathrm{Re}(c^H\mu_d)\big)\,
    \psi\!\big(c^{H}\Gamma_d c + \mathrm{Re}(c^{H}J_d \bar{c})\big),
    \label{eq:CES}
\end{equation}
for some characteristic generator function $\psi$. We write $d \sim CES(\mu_d, \Gamma_d, J_d, \psi)$ to denote this property.
\end{definition}
The special case $\mu_d = 0$ is called the  {centered CES distribution}.  
For any given characteristic generator $\psi$, the CES family is closed under affine transformations, i.e., if $w \sim CES(\mu_w, \Gamma_w, J_w)$, then
\begin{equation}
    Aw + b \;\sim\; CES(A\mu_w + b,\, A\Gamma_w A^{H},\, AJ_w A^{T}),
    \label{eq:affine}
\end{equation}
for any constant matrix $A$ and vector $b$. A CES random vector is said to be  {proper} (or  {second-order circular}) when its pseudo-covariance vanishes, i.e., $J_d = 0$. In this case, the distribution is invariant under multiplication by any complex unit-modulus scalar, meaning its statistical properties depend only on the covariance (or scatter) matrix $\Gamma_d$. Proper CES models, therefore, capture circularly symmetric behavior such as that of the complex normal distribution, while improper cases ($J_d \neq 0$) account for non-circular or correlated real-imaginary components.

\begin{table}[t]
\centering
\small
\caption{Representative subclasses of complex elliptically symmetric (CES) distributions.}
\label{CES}
\begin{tabular}{|l|c|c|c|}
\hline
\textbf{Distribution} & \textbf{Density $g(t)$} & \textbf{Normalizing $C_g$} & \textbf{Special cases} \\
\hline
{Complex student-$t$} 
& $\!\!\left(1+\frac{2t}{\nu}\right)^{-(2m+\nu)/2}\!\!$ 
& $\dfrac{2^{m}\Gamma(\frac{2m+\nu}{2})}{(\pi\nu)^{m}\Gamma(\frac{\nu}{2})}$ 
& $\nu{=}1$: Cauchy\\
${CT}_{\nu,m}({\mu},{\Gamma},J)$ & & & $\nu{\to}\infty$: CG \\
\hline
{Complex generalized} 
& $\exp(-t^{s}/b)$ 
& $\dfrac{s\,\Gamma(m)\,b^{-m/s}}{\pi^{m}\Gamma(m/s)}$ 
& $s{=}1$: CG\\
gaussian ${CGG}_s({\mu},{\Gamma},J)$ & & & $s{=}\frac{1}{2}$: Laplace \\
\hline
{Complex $W$-distribution} 
& $t^{s-1}\exp(-t^{s}/b)$ 
& $\dfrac{s\,\Gamma(m)\,b^{-(s+m-1)/s}}{\pi^{m}\Gamma((s+m-1)/s)}$ 
& $s{=}1$: CG\\
${CW}_{s}({\mu},{\Gamma},J)$ & & &\\
\hline
{Complex $K$-distribution} 
& $t^{\tfrac{\nu-m}{2}}K_{\nu-m}(2\sqrt{\nu t})$ 
& $\dfrac{2\,\nu^{(\nu+m)/2}}{\Gamma(\nu)\pi^{m}}$ 
& $\nu{\to}\infty$: CG\\
$CK_{\nu,m}({\mu},{\Gamma},J)$ & & &\\
\hline
\end{tabular}
\end{table}
In Table~\ref{CES}, we summarize representative subclasses of the CES family corresponding to different choices of the density generator $g(t)$ \cite{6263313}; where $\nu>0$ denotes the degrees of freedom, $s>0$ the exponent, and $b>0$ the scale parameter, while $CG$ represents the complex Gaussian distribution.

For the objective function in \eqref{3CP}, we introduce a new variable $t$, such that the following holds:
\begin{align}
    \min_{z\in \mathbb{C}^n,t\in \mathbb{R}} &t\label{prob CES0}\\
    \text{s.t. }&\mathbb{P}[\mathrm{Re}(c^Hz)\le t]\geq p_0\label{prob CES1}\\
    &\mathbb{P}[\varphi_i(z)\le 0]\geq p_i, \qquad i=1,\cdots,m. \label{prob CES}
\end{align}
\begin{theorem}\label{theorem CES}
Let $\Phi(p)$ represent the CDF of the standard elliptical distribution evaluated at $p$. the probabilistic constraint \eqref{prob CES}
is equivalent to:
\begin{align}
\mu_{\varphi_i(z)}+\Phi^{-1}(p_i)\sigma_{\varphi_i(z)}\leq 0 ,\quad i=1,\cdots,m
\end{align}
Moreover, if the probability constraints are individual, then the problem is a convex second-order cone problem. 
\end{theorem}
\begin{proof}
The constraint is stated as, for $i=1,\cdots,m$:
\begin{align}
    &\mathbb{P}[\varphi_i(z)\le0]=\mathbb{P}[\tilde{\varphi_i}(z)\le-\mu_{\varphi_i(z)}/\sigma_{\varphi_i(z)}]=\Phi(-\mu_{\varphi_i(z)}/\sigma_{\varphi_i(z)})\ge p_i\\
    & ~~~~\Phi\left(-\mu_{\varphi_i(z)}/\sigma_{\varphi_i(z)}\right)\geq p_i\iff \mu_{\varphi_i(z)}+\Phi^{-1}(p_i)\sigma_{\varphi_i(z)}\leq 0\label{23}
\end{align}
Since $\sigma_{\varphi_i(z)}$ is a second-order function, for the individual 3CP, the constraint (\ref{23}) is convex, and the problem \eqref{prob CES0}-\eqref{prob CES} has an equivalent convex SOCP.  
\end{proof}
\section{Distributionally Robust 3CP (DRO)} \label{sec: 5}
In this section, we derive explicit deterministic counterparts of the distributionally robust complex chance constraints. In contrast to the previous section, the exact probability distribution of the random data is not assumed to be known; instead, only partial information is available. By  {distributional robustness}, we mean that the chance constraint
\begin{align*}
    \mathbb{P}\big[\mathrm{Re}(d_i\Tilde{z}) \le 0\big] \ge p_i, \quad i=1,\ldots,m,
\end{align*}
is required to hold uniformly over an entire family $\mathcal{D}$ of probability distributions for the data $d$. Equivalently, we consider enforcing the worst-case (distributionally robust) condition
\begin{align}
    \inf_{d_i \sim \mathcal{D}} 
    \mathbb{P}\big[\mathrm{Re}(d_i\Tilde{z}) \le 0\big] 
    \ge p_i, 
    \quad i=1,\ldots,m,
    \label{eq:DRCCP}
\end{align}
where $d_i \sim \mathcal{D}$ indicates that the distribution of $d_i$ belongs to the family $\mathcal{D}$. 
\subsection{Moment Based}\label{subsec: 5.1} 
We now consider a moment-based ambiguity set, specifically, we study three settings: (i) the mean and covariance are fully known; (ii) the mean is known while the covariance is bounded; and (iii) both the mean and the covariance are specified through bounds.
\subsubsection{Ambiguity Set with Known First Two Order Moments}
The first problem we consider involves the ambiguity family $\mathcal{D}$, defined as the set of all distributions characterized by a prescribed mean $\mu_{d_i}$, covariance $\Gamma_{d_i}$, and pseudo-covariance $J_{d_i}$. We denote this family with $\mathcal{D} =(\mu_{d_i},\Gamma_{d_i}, J_{d_i})$.  

\begin{theorem}\label{theorem: 5.1}
For any $p\in (0,1)$, the distributionally robust 3CP
\begin{align}
 \inf_{d_i\sim (\mu_{d_i},\Gamma_{d_i}, J_{d_i})}\mathbb{P}[\mathrm{Re}(d_i\Tilde{z})\le 0]\ge p_i,\quad i=1,\cdots,m\label{30}
\end{align}
is equivalent to the following constraint:
\begin{align}
    k_{p_i} \sigma_{\varphi_i(z)}+\mu_{\varphi_i(z)}\le0, \qquad k_{p_i}=\sqrt{p_i/(1-p_i)},\quad i=1,\cdots,m\label{known mean and var}
\end{align}
 with $\varphi_i(z)=\mathrm{Re}(d_i\tilde{z})$. For the individual 3CP $\forall p_i\in(0,1)$, the constraint (\ref{known mean and var}) is a convex second-order cone constraint.
\end{theorem}
\begin{proof}
    We write $\varphi_i(z)=\mathrm{Re}(d_i\Tilde{z})=\mu_{\varphi_i(z)}+w\sigma_{\varphi_i(z)}, i=1,\cdots,m$ where $w$ has zero mean and unit variance, then from a classical result of Marshall and Olkin \cite{10.1214/aoms/1177705673}, we have 
\begin{align}
    \sup_{d_i\sim (\mu_{d_i},\Gamma_{d_i}, J_{d_i})}\mathbb{P}[\mathrm{Re}(d_i\Tilde{z})> 0]=\sup_{w\sim (0,1)}\mathbb{P}[w\sigma_{\varphi_i(z)}>-\mu_{\varphi_i(z)}]=\frac{1}{1+q_i^2}\label{olkin}
\end{align}
where $q_i^2=\displaystyle\inf_{w\sigma_{\varphi_i(z)}>-\mu_{\varphi_i(z)}}|w|$. Now, we need to find the value of $w$.  We have two cases:
\begin{enumerate}
    \item If $\mu_{\varphi_i(z)}>0$, then we take $w=0$, since the minimum of $|w|$ at $0$.
    \item If $\mu_{\varphi_i(z)}\le 0$, consider the half-space $\mathcal{H}_{t}=\{w\in\mathbb{R}: w\sigma_{\varphi_i(z)}>t\}$, then the Euclidean distance from the origin to $\mathcal{H}_{t}$ is 
\[dist(0,\mathcal{H}_{t})=\frac{|t|}{|\sigma_{\varphi_i(z)}|}\]
In our case, $t=-\mu_{\varphi_i(z)}$. Then $w$ is the distance from the origin of the hyperplane \(dist(0,\mathcal{H}_{-\mu_{\varphi_i(z)}})=\frac{-\mu_{\varphi_i(z)}}{|\sigma_{\varphi_i(z)}|}\),
\end{enumerate}
Therefore, we have the value of $q^2$ as follows
\begin{align}
q_i^2=\begin{cases}
    0 &,~ \mu_{\varphi_i(z)} > 0\\
    \cfrac{\mu_{\varphi_i(z)}^2}{\sigma_{\varphi_i(z)}^2} & ,~\mu_{\varphi_i(z)} \le 0\\
\end{cases} ,\qquad i=1,\cdots,m
\end{align}
Hence, the constraint (\ref{30}) is satisfied if and only if \[1/(1+q^2)\le1-p_i\iff\mu_{\varphi_i(z)}\le 0, \mu_{\varphi_i(z)}^2\ge\sigma_{\varphi_i(z)}^2p_i/(1-p_i),\]
equivalently, $\iff\sqrt{p_i/(1-p_i)}\sigma_{\varphi_i(z)}\le-\mu_{\varphi_i(z)}$.
\end{proof}
\subsubsection{Special Case of Complex Symmetric Distributions}
Consider the situation where, in addition to the first moment and covariance of $d_{i}$, we know that the distribution of $d$ is symmetric around the mean. We say that $d_i$ is symmetric around its mean $\mu_{d_i}$, when $d_i-\mu_{d_i}$ is symmetric around zero (centrally symmetric), where we define central symmetry as follows:

\begin{definition}[Centrally Symmetric Random Vector] A random vector $\xi \in \mathbb{C}^n$ is  {centrally symmetric} if its distribution $\mu_\xi$ is such that
\[
    \mu_\xi(A) = \mu_\xi(-A), \quad \text{for all Borel sets } A \subseteq \mathbb{C}^n.
\]
where $\mu_\xi(A)$ is the probability that $\xi$ falls inside the set $A$.
\end{definition}
Let $\mathcal{D} = (\mu_{d_i},\Gamma_{d_i}, J_{d_i})_S, i=1,\cdots,m$ denote the family of symmetric distributions having mean $\mu_{d_i}$ and variance $\Gamma_{d_i}$. The following theorem gives an explicit condition for the satisfaction of the complex chance constraint robustly over the family $\mathcal{D} = (\mu_{d_i},\Gamma_{d_i}, J_{d_i})_S$.

\begin{theorem}\label{theorem 5.2}
For any $p\in [0.5,1)$, the symmetric distributionally robust 3CP
\begin{align}
 \inf_{d_i\sim (\mu_{d_i},\Gamma_{d_i}, J_{d_i})_S}\mathbb{P}[\mathrm{Re}(d_i\Tilde{z})\le 0]\ge p_i,\quad i=1,\cdots,m\label{34}
\end{align}
is equivalent to the following constraint:
\begin{align}
    k_{p_i} \sigma_{\varphi_i(z)}+\mu_{\varphi_i(z)}\le0, \qquad k_{p_i}=1/\sqrt{2(1-p_i)},\quad i=1,\cdots,m\label{31}
\end{align}
\end{theorem}
\begin{proof}
    If $d$ is symmetric around $\mu_{d}$, then $\varphi_i(z)$ is symmetric around $\mu_{\varphi_i(z)}$ for all $z$. Therefore, (\ref{34}) is satisfied if
\begin{align}
\sup_{\varphi_i(z)\sim(\mu_{\varphi_i(z)},\sigma^2_{\varphi_i(z)})_S}
\mathbb{P}[\varphi_i(z) > 0] \leq 1-p_i,\quad i=1,\cdots,m\label{symm3cp}
\end{align}
Apply the Chebyshev mean-variance inequality for standardized symmetric distributions (see \cite{ion2023sharp}, Theorem 3.1). Let $Z$ be symmetric with mean 0 and variance 1. For $v\ge0$, the inequality holds
\[\mathbb{P}[w\ge v]\le\frac{1}{2\max\{v,1\}^2}=\frac12\min\left\{\frac{1}{v^2},1\right\}\]
In our case, we have $w=\frac{\varphi_i(z)-\mu_{\varphi_i(z)}}{\sigma_{\varphi_i(z)}},$ and $ v=\frac{-\mu_{\varphi_i(z)}}{\sigma_{\varphi_i(z)}},$ with $\mu_{\varphi_i(z)}\le 0$. We want $\frac12\min\left\{\frac{1}{v^2},1\right\}\le 1-p$. Therefore, for $p\in[0.5,1)$, we have the case 
\[\frac{1}{v^2}\le 1-p\quad\Longrightarrow \quad1\le\frac{1}{\sqrt{2(1-p)}}\le v\Longrightarrow\frac{1}{\sqrt{2(1-p)}}\sigma_{\varphi_i(z)}+\mu_{\varphi_i(z)}\le0\]
\end{proof}
\subsubsection{Ambiguity Set with Unknown Second-Order Moment}\label{subsubsec:unknown_second_order}
Consider the case where the mean of $d_i = d_i^R + \mathrm{i} d_i^I \in \mathbb{C}^{n+1}$ is known and equal to $\mu_{d_i}$, while the covariance of the stacked real--imaginary vector
$\begin{bmatrix} d_i^R\\ d_i^I\end{bmatrix}\in\mathbb{R}^{2(n+1)}$
is only known to satisfy an upper bound $L_{d_i}\succeq 0$. We define the ambiguity set, for $i=1,\cdots,m$
\begin{align}
\mathcal{D}(\mu_{d_i},L_{d_i})
=\left\{F \;\middle|\;
\begin{aligned}
&\mathbb{E}_{F}[d_i] = \mu_{d_i},\\[2pt]
&
\begin{bmatrix}
\Gamma_{d_i^R} & \Gamma_{d_i^R,d_i^I}\\
\Gamma_{d_i^I,d_i^R} & \Gamma_{d_i^I}
\end{bmatrix}
\preceq 
\begin{bmatrix}
\hat{\Gamma}_{d_i^R} & \hat{\Gamma}_{d_i^R,d_i^I}\\
\hat{\Gamma}_{d_i^I,d_i^R} & \hat{\Gamma}_{d_i^I}
\end{bmatrix}
=:L_{d_i}
\end{aligned}
\right\}
\end{align}
From the block components of $L_{d_i}$, define
\begin{align}
\hat{\Gamma}_{d_i}
&=\hat{\Gamma}_{d_i^R}+\hat{\Gamma}_{d_i^I}
+\mathrm{i}\!\left(\hat{\Gamma}_{d_i^I,d_i^R}-\hat{\Gamma}_{d_i^R,d_i^I}\right),\\
\hat{J}_{d_i}
&=\hat{\Gamma}_{d_i^R}-\hat{\Gamma}_{d_i^I}
+\mathrm{i}\!\left(\hat{\Gamma}_{d_i^I,d_i^R}+\hat{\Gamma}_{d_i^R,d_i^I}\right).
\end{align}
Moreover, for $\tilde z\in\mathbb{C}^{n+1}$ define its real stacking $\tilde z_{\mathbb{R}}:=\begin{bmatrix}\mathrm{Re}(\tilde z)& \mathrm{Im}(\tilde z)\end{bmatrix}^T\in\mathbb{R}^{2(n+1)}$, so that the quadratic form satisfies
\begin{equation}\label{eq:L_identity}
\tilde z_{\mathbb{R}}^{T}L_{d_i}\tilde z_{\mathbb{R}}
=\frac{1}{2}\left(\tilde z^{H}\hat{\Gamma}_{d_i}\tilde z+\mathrm{Re}(\tilde z^{T}\hat{J}_{d_i}\tilde z)\right),\quad i=1,\cdots,m.
\end{equation}

\begin{theorem}\label{thm:unknown_second_order}
For any $p_i\in(0,1)$, the distributionally robust 3CP constraint
\begin{align}
\inf_{d_i\sim \mathcal{D}(\mu_{d_i},L_{d_i})}
\mathbb{P}\!\left[\mathrm{Re}(d_i\tilde z)\le 0\right]\ge p_i,\quad i=1,\cdots,m
\label{eq:drcc_unknown_second_order}
\end{align}
is equivalent to the deterministic inequality
\begin{align}
\mathrm{Re}(\mu_{d_i}\tilde z)
+k_p
\sqrt{\frac{1}{2}\left(\tilde z^{H}\hat{\Gamma}_{d_i}\tilde z+\mathrm{Re}(\tilde z^{T}\hat{J}_{d_i}\tilde z)\right)}
\le 0,\quad i=1,\cdots,m
\label{eq:det_unknown_second_order}
\end{align}
where $k_p=\sqrt{\frac{p_i}{1-p_i}},\quad i=1,\cdots,m$
\end{theorem}
\begin{proof}
Let $\varphi_i(\tilde z):=\mathrm{Re}(d_i\tilde z)$. Under $\mathrm{Cov}\!\left(\begin{bmatrix} d_i^R\\ d_i^I\end{bmatrix}\right)\preceq L_{d_i}$, the variance of $\varphi_i(\tilde z)$ is maximized at the largest covariance, hence
\[
\sup_{\Sigma\preceq L_{d_i}}\mathrm{Var}(\varphi_i(\tilde z))
=\tilde z_{\mathbb{R}}^{T}L_{d_i}\tilde z_{\mathbb{R}}
=\frac{1}{2}\!\left(\tilde z^{H}\hat{\Gamma}_{d_i}\tilde z+\mathrm{Re}(\tilde z^{T}\hat{J}_{d_i}\tilde z)\right),\quad i=1,\cdots,m.
\]
Applying Theorem~\ref{theorem: 5.1} with mean $\mu_{\varphi_i(\tilde z)}=\mathrm{Re}(\mu_{d_i}\tilde z)$ and worst-case standard deviation
$\sigma_{\varphi_i(\tilde z)}=\sqrt{\tilde z_{\mathbb{R}}^{T}L_{d_i}\tilde z_{\mathbb{R}}}$
yields \eqref{eq:det_unknown_second_order}.
\end{proof}

\subsubsection{Ambiguity Set with Unknown Moments}\label{subsubsec:unknown_moments}
We now assume that the mean is not exactly known: it lies in an ellipsoid of size $\zeta_{d_i}\ge 0$ centered at a nominal estimate $\hat{\mu}_{d_i}$, and the second-order moment satisfies the same covariance upper bound. The ambiguity set is, for $i=1,\cdots,m$
\begin{align}
\mathcal{D}(\hat{\mu}_{d_i},\zeta_{d_i},L_{d_i})
=\left\{F \;\middle|\;
\begin{aligned}
&\big(\mathbb{E}_{F}[d_i]-\hat{\mu}_{d_i}\big)^{H}\hat{\Gamma}_{d_i}^{-1}\big(\mathbb{E}_{F}[d_i]-\hat{\mu}_{d_i}\big)\le \zeta_{d_i},\\[2pt]
&
\Sigma_{d_i}=\mathrm{Cov}_F\!\left(\begin{bmatrix} d_i^R\\ d_i^I\end{bmatrix}\right)\preceq L_{d_i}.
\end{aligned}
\right\}.
\end{align}

\begin{theorem}\label{thm:unknown_moments}
For any $p_i\in(0,1)$, the distributionally robust 3CP constraint
\begin{align}
\inf_{d_i\sim \mathcal{D}(\hat{\mu}_{d_i},\zeta_{d_i},L_{d_i})}
\mathbb{P}\!\left[\mathrm{Re}(d_i\tilde z)\le 0\right]\ge p_i
\label{eq:drcc_unknown_moments}
\end{align}
is equivalent to
\begin{align}
\mathrm{Re}(\hat{\mu}_{d_i}\tilde z)
+\sqrt{\zeta_{d_i}}\sqrt{\tilde z^{H}\hat{\Gamma}_{d_i}\tilde z}
+\sqrt{\frac{p_i}{1-p_i}}\,
\sqrt{\frac{1}{2}\!\left(\tilde z^{H}\hat{\Gamma}_{d_i}\tilde z+\mathrm{Re}(\tilde z^{T}\hat{J}_{d_i}\tilde z)\right)}
\ \le\ 0.
\label{eq:det_unknown_moments}
\end{align}
for $i=1,\cdots,m$. If $d_i$ is proper (so that $J_{d_i}=0$), then \eqref{eq:det_unknown_moments} reduces to
\begin{align}
\mathrm{Re}(\hat{\mu}_{d_i}\tilde z)
+\left(\sqrt{\zeta_{d_i}}+\sqrt{\frac{p_i}{2(1-p_i)}}\right)\sqrt{\tilde z^{H}\hat{\Gamma}_{d_i}\tilde z}
\ \le\ 0,\quad i=1,\cdots,m.
\end{align}
\end{theorem}
\begin{proof}
Let $Y_i:=\mathrm{Re}(d_i\tilde z)$. For fixed first- and second-order moments, $Y_i$ is a real-valued random variable with mean $m_i=\mathrm{Re}(\mu_{d_i}\tilde z),$ and variance
$s_i^2=\tilde z_{\mathbb R}^T\Sigma_{d_i}\tilde z_{\mathbb R}
=\frac12\Big(\tilde z^H\Gamma_{d_i}\tilde z+\mathrm{Re}(\tilde z^T J_{d_i}\tilde z)\Big)$, where $\tilde z_{\mathbb R}=[\mathrm{Re}(\tilde z)^T\ \ \mathrm{Im}(\tilde z)^T]^T$. Hence, with \[
\mathcal U_i=\Big\{(\mu_{d_i},\Sigma_{d_i}):
(\mu_{d_i}-\hat\mu_{d_i})^H\hat\Gamma_{d_i}^{-1}(\mu_{d_i}-\hat\mu_{d_i})\le \zeta_{d_i},
\ \Sigma_{d_i}\preceq L_{d_i}\Big\},\] we have $\displaystyle\inf_{d_i\sim \mathcal D(\hat\mu_{d_i},\zeta_{d_i},L_{d_i})}
\mathbb P\!\left[\mathrm{Re}(d_i\tilde z)\le 0\right]
=
\inf_{(\mu_{d_i},\Sigma_{d_i})\in\mathcal U_i}
\inf_{F\in\mathcal P(\mu_{d_i},\Sigma_{d_i})}
\mathbb P(Y_i\le 0)$.
\\
By the tight one-sided Chebyshev inequality,
\[
\inf_{F\in\mathcal P(\mu_{d_i},\Sigma_{d_i})}\mathbb P(Y_i\le 0)
=
\begin{cases}
\dfrac{m_i^2}{m_i^2+s_i^2}, & m_i\le 0,\\[1ex]
0, & m_i>0.
\end{cases},\quad i=1,\cdots,m
\]
Therefore, \eqref{eq:drcc_unknown_moments} is equivalent to $\inf_{(\mu_{d_i},\Sigma_{d_i})\in\mathcal U_i}
\frac{-\mathrm{Re}(\mu_{d_i}\tilde z)}
{\sqrt{\tilde z_{\mathbb R}^T\Sigma_{d_i}\tilde z_{\mathbb R}}}
\ge \sqrt{\frac{p_i}{1-p_i}}.$ Since the numerator depends only on $\mu_{d_i}$ and the denominator only on $\Sigma_{d_i}$, this becomes $\frac{g_{1,i}(\tilde z)}{\sqrt{g_{2,i}(\tilde z)}}
\ge \sqrt{\frac{p_i}{1-p_i}},$ where $g_{1,i}(\tilde z)=\min_{\mu_{d_i}}-\mathrm{Re}(\mu_{d_i}\tilde z)
\quad\text{s.t.}\quad(\mu_{d_i}-\hat\mu_{d_i})^H\hat\Gamma_{d_i}^{-1}(\mu_{d_i}-\hat\mu_{d_i})\le \zeta_{d_i},$ and $g_{2,i}(\tilde z)=\max_{\Sigma_{d_i}}\tilde z_{\mathbb R}^T\Sigma_{d_i}\tilde z_{\mathbb R}\quad\text{s.t.}\quad\Sigma_{d_i}\preceq L_{d_i}$.
Using the support function of the ellipsoid,
$g_{1,i}(\tilde z)=-\mathrm{Re}(\hat\mu_{d_i}\tilde z)
-\sqrt{\zeta_{d_i}}\sqrt{\tilde z^H\hat\Gamma_{d_i}\tilde z}$, and since $\Sigma_{d_i}\preceq L_{d_i}$,\[
g_{2,i}(\tilde z)
=
\tilde z_{\mathbb R}^TL_{d_i}\tilde z_{\mathbb R}
=
\frac12\Big(\tilde z^H\hat\Gamma_{d_i}\tilde z+\mathrm{Re}(\tilde z^T\hat J_{d_i}\tilde z)\Big),\quad i=1,\cdots,m.
\]
Substituting these expressions and rearranging yields \eqref{eq:det_unknown_moments}. If $d_i$ is proper, then $\hat J_{d_i}=0$, which gives the simplified form.
\end{proof}
\subsection{Support-Based: Random Data with Norm Bounds}\label{subsec: 5.2}
We now study a model of data uncertainty in which the random vector $d$ has a known mean $\mu_{d}$, and the norm bounds of the data points of $d$. Without loss of generality, consider one constraint $\mathbb{P}[\varphi(z)<0]\ge p$. We write $d$ as:
\[d=\mu_{d}+w,\]
where $w\in\mathbb{C}^{n+1}$ is a zero-mean random vector composed of independent elements. The norm of each component of $w$ is bound by specific value, i.e., $\|w_i\|\le l_i, i=1,\cdots,n+1$. In this case we denote with $(\mu_d,L)_N$ the family of distributions on the $(n + 1)$-dimensional random variable $d$ satisfying the above condition, where $L$ is a diagonal matrix containing the norm bounds,
    \[L=diag(l_1,\cdots,l_{n+1})\]

\begin{theorem}\label{theorem: norm bound}
For any $p\in(0,1)$, if $\mu_{\varphi(z)}\le0$, then the distributionally robust complex chance-constraint: 
\[\inf_{d\in(\mu_d, L)_N}\mathbb{P}[\varphi(z)\le0]\ge p\]
holds if 
\[\sqrt{2\ln(1/(1-p)}\|L\Tilde{z}\|+\mu_{\varphi(z)}\le 0\]
\end{theorem}
\begin{proof}
We express the constraint function as $\varphi(z) = \mu_{\varphi(z)}+ \mathrm{Re}(w^H \tilde{z}),$ where $w$ is complex random variable with zero mean and bounded norm $\|w_i\|\le l_i$. Set $\xi_i = \mathrm{Re}(w_i \tilde{z}_i), i=1,\cdots,n+1$, then:
\[
|\xi_i| = |\mathrm{Re}(w_i \tilde{z}_i)| \leq \|w_i \tilde{z}_i\| \leq l_i \|\tilde{z}_i\|.
\]
Thus, the $\xi_i\in[-a_i,a_i]$ with $a_i = l_i \|\tilde{z}_i\|$ are zero-mean independent, and bounded by intervals of width $2l_i\|\tilde{z}_i\|$, for $i=1,\cdots,n+1$. Hence, we have:
\begin{align}
    \mathbb{P}[\varphi(z)\leq0]=\mathbb{P}\left[\sum_{i=1}^{n+1}\xi_i\le-\mu_{\varphi(z)}\right]\label{36}
\end{align}
 Therefore, applying Hoeffding's inequality (see \cite{409cf137-dbb5-3eb1-8cfe-0743c3dc925f}) to (\ref{36}), with $\mu_{\varphi(z)}\le 0$:
\[
\mathbb{P}[\varphi(z)\le0]\ge 1-\exp\left[\cfrac{-2\mu_{\varphi(z)}^2}{\sum_{i=1}^{n+1}(2a_i)^2}\right]= 1-\exp\left[\cfrac{-\mu_{\varphi(z)}^2}{2\|L\Tilde{z}\|^2}\right].
\]
To ensure that 
\begin{align*}
&1-\exp\left[\cfrac{-\mu_{\varphi(z)}^2}{2\|L\Tilde{z}\|^2}\right]\ge p\iff \exp\left[\cfrac{-\mu_{\varphi(z)}^2}{2\|L\Tilde{z}\|^2}\right]\le 1-p
\cfrac{-\mu_{\varphi(z)}^2}{2\|L\Tilde{z}\|^2}\le \ln(1-p)\\
&\iff \cfrac{\mu_{\varphi(z)}^2}{2\|L\Tilde{z}\|^2}\ge -\ln(1-p)=\ln\frac{1}{1-p}\iff |\mu_{\varphi(z)}|\ge \|L\tilde{z}\|\sqrt{2\ln\frac{1}{1-p}}
\end{align*}
since $\mu_{\varphi(z)}<=0$, then $|\mu_{\varphi(z)}|=-\mu_{\varphi(z)}$, and the result follows.
\end{proof}
\subsection{Data-Driven Based: Empirical Estimation}
Suppose we have a batch of independent samples $d_1, \dots, d_N$ drawn from an unknown complex distribution, with the column vector $d = [a ~ b]$. Then the standard empirical estimates of the true mean $\mu_d$, covariance $\Gamma_d$, and pseudo-covariance $J_d$ are given by
\begin{align}
    &\widehat{\mu_d}=\frac{1}{N}\sum_{i=1}^N d_i,\\
    &\widehat{\Gamma_d}=\frac{1}{N}\sum_{i=1}^N (d_i-\widehat{\mu_d})(d_i-\widehat{\mu_d})^H\\
    &\widehat{J_d}=\frac{1}{N}\sum_{i=1}^N (d_i-\widehat{\mu_d})(d_i-\widehat{\mu_d})^T
\end{align}
\begin{lemma} \label{lemma: 6.1}Let $d_1,\cdots,d_N \in\mathbb{C}^{n+1}$ be extracted independently according to the unknown complex distribution $K$ and let 
\begin{align}
    R =\sup_{d\in \text{support}(K)}\|d\|, \quad\beta\in(0,1).
\end{align}
Then, with probability at least $1-\beta$, it holds that
\begin{align}
    \|\widehat{\mu_d}-\mu_d\|\le (R/\sqrt{N})\left[2+\sqrt{2\log(1/\beta)}\right]\label{mean}
\end{align}
If, provided that $N\ge\left[2+\sqrt{2\log(2/\beta)}\right]^2$ with probability at least $1-\beta$, it holds that
\begin{align}
    \|\widehat{\Gamma_d}-\Gamma_d\|\le (2R^2/\sqrt{N})\left[2+\sqrt{2\log(2/\beta)}\right]\\
    \|\widehat{J_d}-J_d\|\le(2R^2/\sqrt{N})\left[2+\sqrt{2\log(2/\beta)}\right]\label{43}
\end{align}
\end{lemma}
\begin{proof}
For the first two parts (mean and covariance), see \cite{calafiore2006distributionally}.  
We now prove the inequality of the pseudo-covariance. Define
\[
\widetilde{J}_d \;=\; \frac{1}{N}\sum_{i=1}^N (d_i-\mu_d)(d_i-\mu_d)^T.
\]
First, we have
\begin{align*}
&\widehat{J}_d
= \frac{1}{N}\sum_{i=1}^N (d_i-\widehat{\mu}_d)(d_i-\widehat{\mu}_d)^T \\
&= \frac{1}{N}\sum_{i=1}^N \big[(d_i-\mu_d)+(\mu_d-\widehat{\mu}_d)\big]\big[(d_i-\mu_d)+(\mu_d-\widehat{\mu}_d)\big]^T \\
&= \widetilde{J}_d + (\mu_d-\widehat{\mu}_d)(\mu_d-\widehat{\mu}_d)^T + \frac{1}{N}\sum_{i=1}^N \Big[(d_i-\mu_d)(\mu_d-\widehat{\mu}_d)^T + (\mu_d-\widehat{\mu}_d)(d_i-\mu_d)^T\Big]
\end{align*}
Since $\tfrac{1}{N}\sum_{i=1}^N (d_i-\mu_d) = \widehat{\mu}_d-\mu_d$, the cross terms simplify to
\[
(\widehat{\mu}_d-\mu_d)(\mu_d-\widehat{\mu}_d)^T + (\mu_d-\widehat{\mu}_d)(\widehat{\mu}_d-\mu_d)^T
= -2\,(\widehat{\mu}_d-\mu_d)(\widehat{\mu}_d-\mu_d)^T.
\]
Thus, $\widehat{J}_d \;=\; \widetilde{J}_d - (\widehat{\mu}_d-\mu_d)(\widehat{\mu}_d-\mu_d)^T
\Longrightarrow\widehat{J}_d - J_d \;=\; (\widetilde{J}_d - J_d) - (\widehat{\mu}_d-\mu_d)(\widehat{\mu}_d-\mu_d)^T.$ Taking Frobenius norms,
\[
\|\widehat{J}_d - J_d\|_F \;\le\; 
\underbrace{\|\widetilde{J}_d - J_d\|_F}_{\text{(A)}} 
+ \underbrace{\|(\widehat{\mu}_d-\mu_d)(\widehat{\mu}_d-\mu_d)^T\|_F}_{\text{(B)}}.
\]
We need to find bounds for (A) and (B). For (A), each $\|d_i\|\le R$ and then, $\|d_i d_i^T\|_F=\|d_i\|^2\le R^2$. Applying (\ref{mean}) inequality with failure probability $\beta/2$ yields
\[
\|\widetilde{J}_d - J_d\|_F=\Big\|\frac{1}{N}\sum d_i d_i^T-\mathbb{E}[dd^T]\Big\|_F \;\le\; \frac{R^2}{\sqrt{N}}\Bigl(2+\sqrt{2\log(2/\beta)}\Bigr).
\]
For bound (B), we have $\|(\widehat{\mu}_d-\mu_d)(\widehat{\mu}_d-\mu_d)^T\|_F = \|\widehat{\mu}_d-\mu_d\|^2.$ By the mean bound (with failure probability $\beta/2$),
\[
\|\widehat{\mu}_d-\mu_d\| \;\le\; \frac{R}{\sqrt{N}}\Bigl(2+\sqrt{2\log(2/\beta)}\Bigr),
\]
\[
\Longrightarrow\text{(B)} \;\le\; \frac{R^2}{N}\Bigl(2+\sqrt{2\log(2/\beta)}\Bigr)^2
\;\le\; \frac{R^2}{\sqrt{N}}\Bigl(2+\sqrt{2\log(2/\beta)}\Bigr),
\]
since $N \ge \bigl(2+\sqrt{2\log(2/\beta)}\bigr)^2$. By the union bound, (A) and (B) hold simultaneously with probability at least $1-\beta$. Therefore,
\[
\|\widehat{J}_d - J_d\|_F \;\le\; 
\frac{2R^2}{\sqrt{N}}\Bigl(2+\sqrt{2\log(2/\beta)}\Bigr),
\]
\end{proof}

\begin{theorem}\label{theorem data}
Let $d_1,\ldots,d_N$ be $N$ independent samples of the random vector 
$d \in \mathbb{C}^{n+1}$ having an unknown distribution $K$, and let $R = \sup_{d \in \mathrm{support}(K)} \|d\|$. Let $\widehat{d}_N, \widehat{\Gamma_d},$ and $ \widehat{J_d}$ be the sample estimates of the mean, covariance, and pseudo-covariance of $d$, computed based on the $N$ available samples, and denote by $\mu_d, \Gamma_d$ and $J_d$ the respective true (unknown) values. Define
\[
r_1 := \frac{R}{\sqrt{N}}\left(2+\sqrt{2\log\left(\tfrac{6}{\delta}\right)}\right), 
\qquad
r_2 := \frac{2R^2}{\sqrt{N}}\left(2+\sqrt{2\log\left(\tfrac{6}{\delta}\right)}\right).
\]
Then, for assigned probability levels $p,\delta \in (0,1)$, the distributionally robust complex chance constraint
\begin{align}
\inf_{d \sim (\widehat{\mu_d},\widehat{\Gamma_d}, \widehat{J_d})} 
\mathbb{P}[\mathrm{Re}(d^H \tilde{z}) \leq 0] \;\geq\; p\label{c3CP}
\end{align}
holds with probability at least $1-\delta$, provided that $
N \;\geq\; \left(2+\sqrt{2\log\left(\tfrac{6}{\delta}\right)}\right)^2,$ and
\begin{align*}
    \sqrt{\frac{p}{1-p}}\sqrt{\frac{1}{2}(\tilde{z}^H(\widehat{\Gamma_d}+r_2I)\tilde{z}+\mathrm{Re}(\tilde{z}^T(\widehat{J_d}+r_2I)\tilde{z}))} + \mathrm{Re}(\widehat{\mu_d}^H\tilde{z})+ \|\tilde{z}\|r_1\le 0
\end{align*}
\end{theorem}
\begin{proof}
    Applying Lemma~\ref{lemma: 6.1} with $\beta = \delta/3$, and $        N \ge \left[2+\sqrt{2\log\left(\tfrac{6}{\delta}\right)}\right]^2,$ we have that, with probability at least $1-\delta/3$, it holds that:
    \begin{align}
        \mu_d &= \widehat{\mu_d}+\xi, 
        &&\text{for some }\xi\in\mathbb{C}^{n+1}:\;\|\xi\|\le r_1, \label{45}\\
        \Gamma_d &= \widehat{\Gamma_d}+\Delta, 
        &&\text{for some }\Delta\in\mathbb{C}^{(n+1)\times(n+1)}:\;\|\Delta\|\le r_2, \label{46}\\
        J_d &= \widehat{J_d}+\Delta, 
        &&\text{for some }\Delta\in\mathbb{C}^{(n+1)\times(n+1)}:\;\|\Delta\|\le r_2. \label{47}
    \end{align}
    Combining the three events above, we have that equations~\eqref{45}, \eqref{46}, and~\eqref{47} jointly hold with probability at least $(1-\delta/3)^3 \ge 1-\delta.$ From Theorem~\ref{theorem: 5.1}
    \begin{align}
        k_p\sqrt{\tilde{z}^H\Gamma_d\tilde{z} + \mathrm{Re}(z^T J_d \tilde{z})} 
        + \mathrm{Re}(\mu_d^H\tilde{z}) \le 0, 
        \quad k_p=\sqrt{\tfrac{p}{1-p}}, \label{48}
    \end{align}
    implies the satisfaction of the complex chance constraint~\eqref{c3CP}. 
    Hence, substituting~\eqref{45}–\eqref{46} into~\eqref{48}, we have that~\eqref{c3CP} holds with probability at least $1-\delta$, provided that the following inequality holds:
    \begin{align}
        k_p\sqrt{\tilde{z}^H(\widehat{\Gamma_d}+\Delta)\tilde{z} + \mathrm{Re}(z^T(\widehat{J_d}+\Delta)\tilde{z})} 
        + \mathrm{Re}((\widehat{\mu_d}+\xi)^H\tilde{z}) \le 0, \label{49}
    \end{align}
    for all $\xi:\|\xi\|\le r_1$ and all $\|\Delta\|_F\le r_2$. 
    Then
    \begin{align*}
        & k_p\sqrt{\tilde{z}^H(\widehat{\Gamma_d}+\Delta)\tilde{z} + \mathrm{Re}(z^T(\widehat{J_d}+\Delta)\tilde{z})} 
        + \mathrm{Re}((\widehat{\mu_d}+\xi)^H\tilde{z})\\[2mm]
        &= k_p\sqrt{\tilde{z}^H\widehat{\Gamma_d}\tilde{z} + \mathrm{trace}(\Delta\tilde{z}\tilde{z}^H)
        + \mathrm{Re}(z^T\widehat{J_d}\tilde{z} + \mathrm{trace}(\Delta\tilde{z}\tilde{z}^T))} 
        + \mathrm{Re}((\widehat{\mu_d}+\xi)^H\tilde{z})\\[2mm]
        &\le k_p\sqrt{\tilde{z}^H\widehat{\Gamma_d}\tilde{z} 
        + \|\Delta\|_F\|\tilde{z}\tilde{z}^H\|_F 
        + \mathrm{Re}(z^T\widehat{J_d}\tilde{z}) 
        + \|\Delta\|_F\|\tilde{z}\tilde{z}^T\|_F} 
        + \mathrm{Re}(\widehat{\mu_d}^H\tilde{z})\\&\qquad + \|\xi\|\,\|\tilde{z}\|\\[2mm]
        &\le k_p\sqrt{\tilde{z}^H(\widehat{\Gamma_d}+r_2I)\tilde{z} 
        + \mathrm{Re}(z^T(\widehat{J_d}+r_2I)\tilde{z})} 
        + \mathrm{Re}(\widehat{\mu_d}^H\tilde{z}) + \|\tilde{z}\|\,r_1.
    \end{align*}
\end{proof}
\section{SOCP Approximation of the Joint Probability Constraint}\label{sec: 6}
In the previous sections, we considered the complex chance constraint under different distributional assumptions, expressed in the general form:
\begin{align}
        \min & \;\mathrm{Re}(c^Hz) \label{50}\\[2mm]
        \text{s.t. } 
        & k_{p}(y_i)\sqrt{\frac12(\tilde{z}^H\Gamma_{d_i}\tilde{z} + \mathrm{Re}(\tilde{z}^T J_{d_i}\tilde{z}))}
        + \mathrm{Re}(\mu_{d_i}^H\tilde{z})\le 0,\quad i=1,\cdots,m \label{51}\\[2mm]
        & \sum_{i=1}^my_i=1,\quad y>0,\quad \mathrm{Re}(z)\ge0, \quad\mathrm{Im}(z)\ge0\label{52}
\end{align}
\begin{table}[t]
\centering
\small
\caption{Summary of $k_p(y_i)$, its derivative $k_p'(y_i)$, and valid $p$-ranges for different distributional settings.}\label{tab2: all}
\begin{tabular}{|c|c|c|c|}
\hline
\textbf{Case} & \textbf{$p$-range} & \textbf{$k_p(y_i)$} & \textbf{$k_p'(y_i)$} \\[3pt]
\hline
{CES distribution} 
& $[0.5,1)$ 
& $\Phi^{-1}\!\left(p^{\,y_i^{\tfrac{1}{\theta}}}\right)$ 
& $\tfrac{p^{\,y_i^{\tfrac{1}{\theta}}}\ln(p)}{\theta\,y_i^{1-\tfrac{1}{\theta}}\,f_{\Phi}\!\big(k_p(y_i)\big)}$ \\[6pt]
\hline
{Moment-Based} 
& $(0,1)$ 
& $ \sqrt{\tfrac{p^{\,y_i^{\tfrac{1}{\theta}}}}{1-p^{\,y_i^{\frac{1}{\theta}}}}}$ 
&  $\tfrac{p^{y_{i}^{\frac{1}{\theta}}} \ln\left(p\right) \, y_{i}^{\frac{1}{\theta} - 1}}{2 \sqrt{\frac{p^{y_{i}^{\frac{1}{\theta}}}}{1 - p^{y_{i}^{\frac{1}{\theta}}}}} \left(p^{y_{i}^{\frac{1}{\theta}}} - 1\right)^{2} {\theta}}$ \\[6pt]
\hline
Symmetric $\mu$ 
& $[0.5,1)$ 
& $ \frac{1}{\sqrt{2(1-p^{\,y_i^{\frac{1}{\theta}}})}}$ 
& $\frac{p^{y_{i}^{\frac{1}{\theta}}} \ln\left(p\right) \, y_{i}^{\frac{1}{\theta} - 1}}{2^{\frac{3}{2}} \left(1 - p^{y_{i}^{\frac{1}{\theta}}}\right)^{\frac{3}{2}} {\theta}}$ \\[6pt]
\hline
$\begin{matrix}
    \text{Known } \mu, \\\\ \text{norm bounds}
\end{matrix}$
& $(0,1)$ 
& $ \sqrt{-2\ln\!\left({1-p^{\,y_i^{\frac{1}{\theta}}}}\right)}$ 
& $\frac{p^{y_{i}^{\frac{1}{\theta}}} \ln\left(p\right) \, y_{i}^{\frac{1}{\theta} - 1}}{\sqrt{2} \sqrt{-\ln\left(1 - p^{y_{i}^{\frac{1}{\theta}}}\right)} \left(1 - p^{y_{i}^{\frac{1}{\theta}}}\right) {\theta}}$ \\[6pt]
\hline
{Data-Driven} 
& $(0,1)$ 
& $ \sqrt{\frac{p^{\,y_i^{\frac{1}{\theta}}}}{1-p^{\,y_i^{\frac{1}{\theta}}}}}$ 
& $ \frac{p^{y_{i}^{\frac{1}{\theta}}} \ln\left(p\right) \, y_{i}^{\frac{1}{\theta} - 1}}{2 \sqrt{\frac{p^{y_{i}^{\frac{1}{\theta}}}}{1 - p^{y_{i}^{\frac{1}{\theta}}}}} \left(p^{y_{i}^{\frac{1}{\theta}}} - 1\right)^{2} {\theta}}$\\[3pt]
\hline
\end{tabular}
\end{table}
The function $k_{p}(y_i)$ depends on the underlying distribution and the assigned probability level $p$ as summarized in Table \ref{tab2: all}. For the individual 3CP, the probability level $p_i$ is fixed; hence, $k_p(y_i)$ is constant, and each constraint in~\eqref{51} is a convex second-order cone constraint. However, in most practical settings, one typically works with the joint 3CP formulation introduced in Section \ref{sec: 3}. In this case, the feasible set associated with~\eqref{51} is no longer convex. To address this issue, we adapt the bounding techniques of~\cite{CHENG2012325} in the following subsections to develop a SOCP approximation of~\eqref{50}--\eqref{52}. Specifically, we derive both an upper bound and a feasible lower bound for the function $k_p(y_i)$.\\
\textbf{Lower Bound Approximation via Taylor Expansion:}\\
To formulate a deterministic approximation providing a lower-bound solution to~\eqref{50}--\eqref{52}, we construct a piecewise linear under estimator of $k_p(y_i)$, namely a function that lies below $k_p(y_i)$ over its domain, using first-order Taylor expansions around $N$ tangent points $t_l \in (0,1]$, $l=1,\ldots,N$, with $t_1 < t_2 < \cdots < t_N$. For each tangent point, we define:
\begin{align}
    \hat{K}_{1l} &= k_p(t_l) + (y - t_l)k_p'(t_l) = \alpha_{1l} + \beta_{1l}y,\\
    \beta_{1l} &= k_p'(t_l), \quad \alpha_{1l} = k_p(t_l) - \beta_{1l}t_l.
\end{align}
The derivative $k_p'(r_l)$ depends on the specific distributional case taken from Table \ref{tab2: all}. The lower approximation function is given by $\hat{K}_1(y) =\displaystyle \max_{l=1,\ldots,N} \hat{K}_{1l}(y).$ Since $k_p(y)$ is convex, each linearization $\hat{K}_{1l}$ is a global under estimator, and their maximum remains a valid lower bound.\\ 
\textbf{Upper Bound Approximation via Piecewise Linear Interpolation:}\\
To construct an upper approximation to~\eqref{50}--\eqref{52}, we select $N$ interpolation points 
$t_l\in(0,1]$, $l=1,\ldots,N$, with $t_1<t_2<\cdots<t_N$, and denote $K_l=k_p(t_l)$. 
Let $\hat{K}_{2l}$ be the corresponding piecewise linear interpolation of $k_p(y)$:
\begin{align}
    &\hat{K}_{2l} = K_l + \frac{y-t_l}{t_{l+1}-t_l}(K_{l+1}-K_l)
    = \alpha_{2l}+\beta_{2l}y,\\
&\alpha_{2l}=\frac{t_{l+1}K_{l}-t_lK_{l+1}}{t_{l+1}-t_l}, \qquad 
\beta_{2l}=\frac{K_{l+1}-K_l}{t_{l+1}-t_l}.
\end{align}
Define the upper approximation function as $\hat{K}_2(y)=\max_{l=1,\ldots,N}\hat{K}_{2l}(y)$.
\begin{theorem}\label{thm:5.1}
Let $r_i=(r_{i1},\ldots,r_{in})\in\mathbb{C}^n$ and $w_i=(w_{i1},\ldots,w_{in})\in\mathbb{C}^n$, for $i=1,\ldots,m$. For each $l$, let $(\alpha_l,\beta_l)$ be either the tangent coefficients $(\alpha_{1l},\beta_{1l})$ or the interpolation coefficients $(\alpha_{2l},\beta_{2l})$ of $k_p(y)$. Consider the following deterministic SOCP approximation of problem~\eqref{50}--\eqref{52}:
\begin{equation}
\begin{aligned}
&\min_{z,\{r_i\},\{w_i\}\in\mathbb{C}^n}\quad \mathrm{Re}(c^Hz)\\
&\text{s.t.}\quad 
 \mathrm{Re}(\mu_{a_i}z) + \sqrt{\tfrac{1}{2}(r_i^H\Gamma_{a_i}r_i+\mathrm{Re}(r_i^TJ_{a_i}r_i))+\sigma_{b_i}}-\mu_{b_i}\leq0,\quad i=1,\cdots,m,\\
& \mathrm{Re}(r_{ij}) \geq \mathrm{Re}(\alpha z_j + \beta w_{ij}),~\mathrm{Im}(r_{ij}) \geq\mathrm{Im}(\alpha z_j + \beta w_{ij}), i=1,\cdots,m,j=1\cdots,n,\\
& \sum_{i=1}^m w_{ij}=z_j,\quad \mathrm{Re}(z_j)\ge0, \quad\mathrm{Im}(z_j)\ge0\quad j=1,\cdots,n.
\end{aligned}
\label{7.7}
\end{equation}
Then:
\begin{itemize}
\item[(i)] If $\hat K(y)=\hat K_1(y)$ is the tangent-based piecewise linear under estimator of $k_p(y)$, then the optimal value of \eqref{7.7} provides a lower bound to the original problem~\eqref{50}--\eqref{52}.
\item[(ii)] If $\hat K(y)=\hat K_2(y)$ is the piecewise linear interpolation over estimator of $k_p(y)$, then the optimal value of \eqref{7.7} provides an upper bound to the original problem~\eqref{50}--\eqref{52}, under the unified condition $p^{\sum_i y_i^*}\;\leq\;\widehat{k}^{-1}\left(\tfrac{-\mu_{\mathrm{Re}(a_i z)-b_i}}{\sigma_{\mathrm{Re}(a_i z)-b_i}}\right)$, where $\widehat{k}(p^{\,y}) := k_p(y)$ and $\widehat{k}^{-1}$ denotes its inverse on $(0,1)$.
\end{itemize}
\end{theorem}
\begin{proof}
Let $s_i(z):=\sqrt{\tfrac{1}{2}(z^H\Gamma_{a_i}z+\mathrm{Re}(z^TJ_{a_i}z))+\sigma_{b_i}},
\qquad s_i(z)\ge 0$.\\
If $\hat K(y)=\hat K_1(y)$ is the tangent-based approximation, then by convexity of $k_p(y)$, $\hat K_1(y)=\max_{l=1,\ldots,N}\hat K_{1l}(y)\le k_p(y)$. Hence, for each $i$, $\Big\{z:\mathrm{Re}(\mu_{a_i}z)+k_p(y_i)s_i(z)\le \mu_{b_i}\Big\}
\subseteq \Big\{z:\mathrm{Re}(\mu_{a_i}z)+\hat K_1(y_i)s_i(z)\le \mu_{b_i}\Big\}.$
Therefore, the feasible set of \eqref{7.7} is an outer approximation of the original feasible set, and since the problem is a minimization problem, its optimal value provides a lower bound to the optimal value of problem~\eqref{50}--\eqref{52}. If $\hat K(y)=\hat K_2(y)$ is the piecewise linear interpolation approximation, then by convexity of $k_p(y)$, $\hat K_2(y)\ge k_p(y)$.
Thus, for each $i$,
\[
\Big\{z:\mathrm{Re}(\mu_{a_i}z)+\hat K_2(y_i)s_i(z)\le \mu_{b_i}\Big\}
\subseteq
\Big\{z:\mathrm{Re}(\mu_{a_i}z)+k_p(y_i)s_i(z)\le \mu_{b_i}\Big\}.
\]
Hence, the feasible set of \eqref{7.7} is an inner approximation of the original feasible set, and its optimal value provides an upper bound to problem~\eqref{50}--\eqref{52}, under the stated unified condition.
\end{proof}
\section{Numerical experiments} \label{sec: 7}
The experiments were conducted on a machine with an 11th-generation Intel Core i7-1185G7 processor operating at 3.00 GHz and 32 GB of RAM. 
\begin{figure}
    \centering
    \includegraphics[width=1\linewidth]{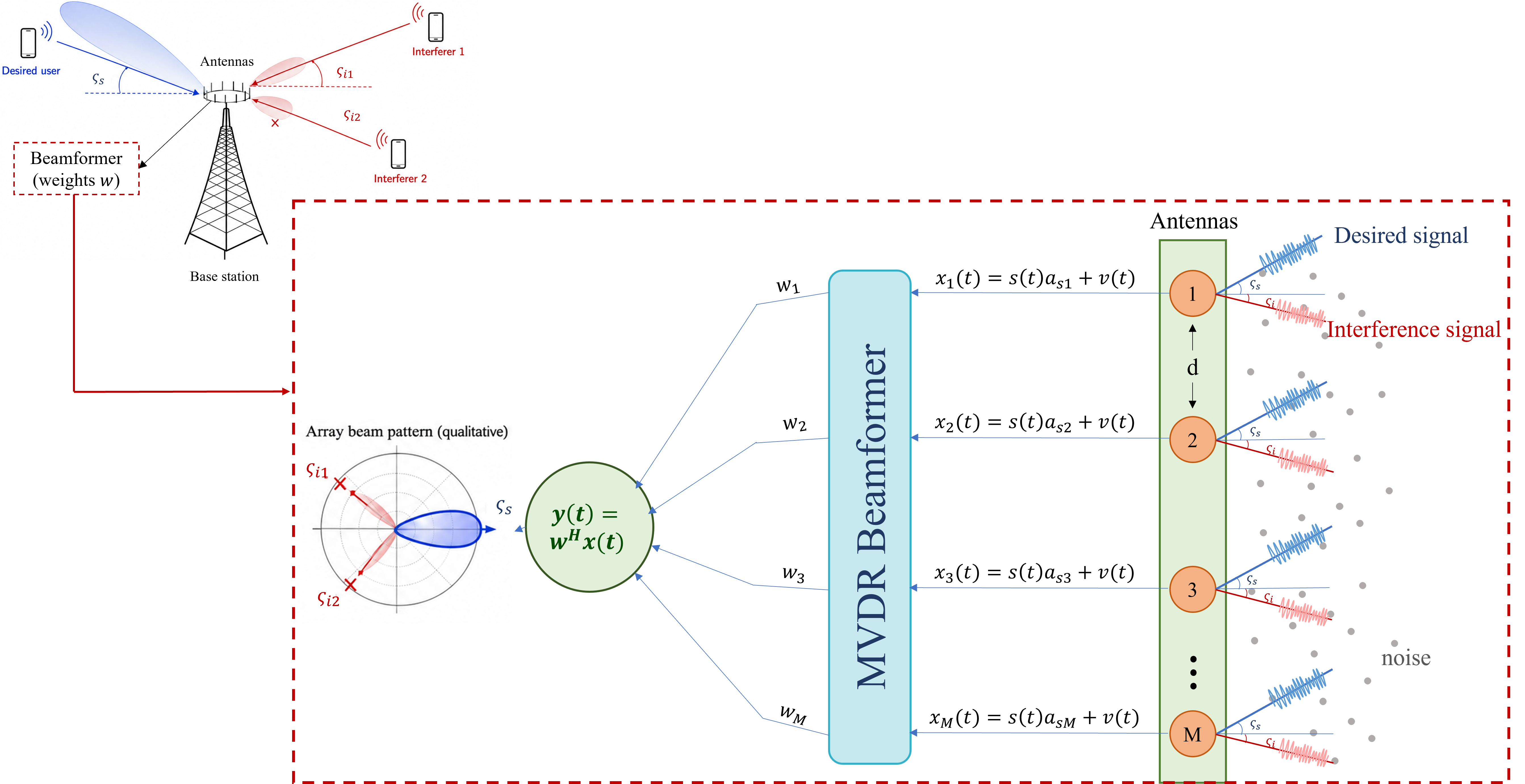}
    \caption{Practical interpretation of MVDR beamforming at a base station.}
    \label{fig:mvdrapp}
\end{figure}

We start by proposing a formalism to the classical adaptive beamforming problem based on the Minimum Variance Distortionless Response (MVDR) criterion~\cite{1165054}. Assume that we have $N$ sensors, and the output of the narrowband beamformer is given by $y(t) = w^H x(t),$ where $t$ is the sample index, $w \in \mathbb{C}^M$ is the beamformer weight vector, and $x(t) = S(t) + v(t)$ is the received snapshot vector composed of the desired signal $S(t) = s(t)a_s$ and the interference-plus-noise component $v(t)$. Here, $s(t)$ and $a_s$ denote the desired signal waveform and its steering vector, respectively. The optimal MVDR beamformer maximizes the signal-to-interference-plus-noise ratio (SINR) $ = \frac{w^H R_s w}{w^H R_{i+n} w},$ where $R_s$ and $R_{i+n}$ denote the covariance matrices of the signal and the interference-plus-noise, respectively. 
Maximizing SINR is equivalent to solving
\begin{equation}
    \min_{w} \; w^H R_{i+n} w 
    \quad \text{s.t.} \quad w^H a_s = 1,
\end{equation}
whose analytical solution is $w^* = R_{i+n}^{-1}a_s / (a_s^H R_{i+n}^{-1}a_s)$, 
yielding $\text{SINR} = \sigma_s^2 a_s^H R_{i+n}^{-1}a_s$. As illustrated in Figure~\ref{fig:mvdrapp}, the received signal is collected by a sensor array and processed by the MVDR beamformer in order to preserve the desired signal while suppressing interference and noise. In practice, the presumed steering vector $a_s$ is often mismatched from the true one due to model errors. We model this mismatch as $\tilde{a}_s = a_s + \delta$, where $\delta$ is a random perturbation. 
The resulting chance-constrained 3CP problem becomes:
\begin{align}
    \min &~~ w^H R_{i+n} w, \label{eq:beam_3CP_obj} \\
    \text{s.t. } &~\mathbb{P}\!\left[-\mathrm{Re}(\delta^H w) \le \mathrm{Re}(a_s^H w) - 1\right] \ge p,  \mathrm{Re}(a_s^H w) \ge 0,~~ \mathrm{Im}(a_s^H w) = 0, \label{eq:beam_3CP_cons}
\end{align}
where $\delta$ is assumed to be a circular complex Gaussian random vector with zero mean. The probabilistic constraint~\eqref{eq:beam_3CP_cons} 
can be equivalently written in deterministic form:
\begin{align}
    \min &~~ w^H R_{i+n} w, \\
    \text{s.t. } &~ k_p\,\|\Gamma_\delta^{1/2} w\| / 2 \le \mathrm{Re}(a_s^H w) - 1,\mathrm{Re}(a_s^H w) \ge 0,~~ \mathrm{Im}(a_s^H w) = 0. \label{MVDR-3CP}
\end{align}

We consider three distributional settings for $\delta$. True complex Gaussian case $\delta \sim CG(0,\Gamma_\delta, 0)$ with known full distribution. Then the true distribution is unknown, but its first and second moments $(0, \Gamma_d,0)$ are known. The deterministic counterpart is obtained using Theorem~\ref{theorem: 5.1}. Finally, the mean and covariance $(\mu_d, \Gamma_d, J_d)$ are empirically estimated from $N=100,000$ Monte Carlo samples, leading to $(\widehat{\mu}_d, \widehat{\Gamma}_d, \widehat{J}_d)$. In this case, constraint \eqref{MVDR-3CP} written as:
\[k_p\sqrt{\frac{1}{2}\left(w^H(\widehat{\Gamma_\delta}+r_2I) w+\mathrm{Re}(w^T(\widehat{J_\delta}+r_2I)w\right)} +\mathrm{Re}(\widehat{\mu_\delta}w)+r_1\|w\|\le \mathrm{Re}(a_s^H w) - 1.\]
\begin{figure}[t]
    \centering
    \includegraphics[width=1\linewidth]{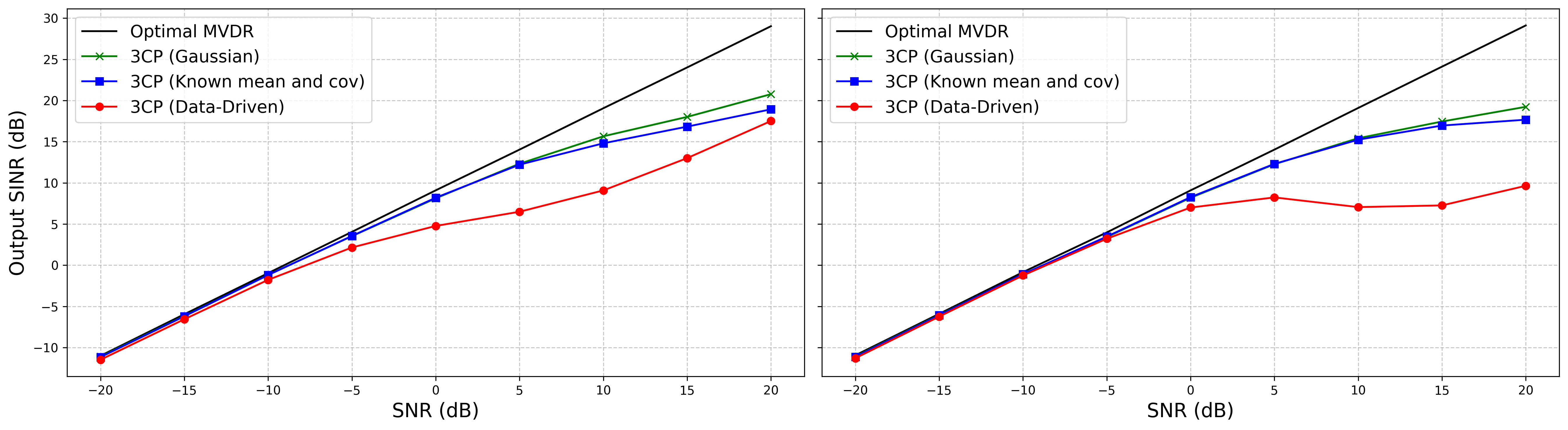}
    \caption{Output SINR versus input SNR for INR = 15 (left) and INR = 25 (right) 
    under three settings: true Gaussian distribution, known $(\mu_d,\Gamma_d,J_d)$, 
    and empirical estimation uncertainty.}
    \label{fig:SINRvsSNR}
\end{figure}
Simulation parameters are set as follows: the number of sensors $M=8$, number of samples $K=100$, inter-element spacing $d=0.5$ wavelengths, and probability level $p=0.7$. The Direction of Arrival (DOA) of the desired signal is $\varsigma_s = 3^\circ$, and two interfering sources are located at $15^\circ$ and $30^\circ$. The perturbation variance is $\sigma_\delta^2 = 1M$ with covariance $\mathrm{Cov}_\delta = \sigma_\delta^2/M\,I_M$. The noise power is $\sigma_n^2 = 1$, and the results are averaged over 100 Monte Carlo runs. Figure~\ref{fig:SINRvsSNR} compares the resulting SINR performance under the three cases. The empirical estimation case gradually converges to the known-moment case as the number of samples increases, and both closely track the true Gaussian distribution, confirming that the proposed 3CP formulation is robust to distributional and estimation uncertainty.

Since the previous application considers a single constraint, we now move to more general cases in order to demonstrate the effectiveness of our model. We generating a random mean vector $\mu_d$, a positive semidefinite covariance matrix $\Gamma_d$, and a symmetric pseudo-covariance matrix $J_d$. Hence, the random vector $d$ follows a distribution $K(\mu_d,\Gamma_d,J_d)$, where $K$ denotes a generic family of complex-valued distributions. We first solve problem~\eqref{3CP} under the assumption that $K$ belongs to the class of complex elliptically symmetric (CES) distributions, considering multiple subtypes as characterized in Theorem~\ref{theorem CES}. Next, we address the case where the exact distribution of $d$ is unknown, but its first- and second-order statistics $(\mu_d,\Gamma_d, J_d)$ are available, and we solve the problem using Theorem~\ref{theorem: 5.1}. We then specialize to the symmetric-mean case, handled by Theorem~\ref{theorem 5.2}. Subsequently, we assume that each component of $d$ is bounded in norm by $|d_i|\le l_i$ with $l_i=10$, and we solve this case using Theorem~\ref{theorem: norm bound}. To analyze the estimation uncertainty, we generate $N=10{,}000$ samples from the original distribution and compute the empirical estimates $(\widehat{\mu}_d, \widehat{\Gamma}_d, \widehat{J}_d)$, which are then used to solve the problem via Theorem~\ref{theorem data}. Finally, we assess the out of sample of the optimal solution $z^\star$. For each case, we generate $1000$ additional scenarios sampled from the original distribution and evaluate the frequency with which $\mathrm{Re}(A z - b)\le 0$ holds. 

The simulation results, summarized in Table~\ref{tab:3CP_results}, with $100$ decision variables and $30$ chance constraints, with fixed $\theta\in\{2,5\}$, are presented for both the individual and joint chance-constraint formulations. For the joint case, we employ the lower-bound SOCP approximation using 20 tangent points. As the target probability level $p$ increases, the optimal objective value increases, and the violation rate decreases across all cases. Importantly, the observed mean violation probability consistently remains below the target level $1-p$, validating the correctness of our theoretical guarantees.
\begin{table}[t]
\centering
\small
\caption{Comparison of Individual and Joint 3CP across distributions and probabilities for $n=50$, $m=15$, and $\theta\in\{2,5\}$.}
\begin{tabular}{|l|c|rr|rr|rr|}
\hline
\textbf{Distribution} & $p$ 
& \multicolumn{2}{c|}{\textbf{Individual}} 
& \multicolumn{2}{c|}{\textbf{Joint ($\theta=2$)}} 
& \multicolumn{2}{c|}{\textbf{Joint ($\theta=5$)}} \\[2pt]
\cline{3-8}
& & Value & Viol.(\%) & Value & Viol.(\%) & Value & Viol.(\%) \\
\hline
\multirow{4}{*}{Gaussian} 
& 0.70 & -9.03 & 18.00 & -4.43 & 5.38 & -6.32 & 10.00 \\
& 0.80 & -5.14 & 7.81 & -3.54 & 3.24 & -4.44 & 5.50 \\
& 0.95 & -2.52 & 1.18 & -2.26 & 0.65 & -2.44 & 0.94 \\
\hline
\multirow{4}{*}{Student-$t$} 
& 0.70 & -8.67 & 18.27 & -4.16 & 5.58 & -6.05 & 10.40 \\
& 0.80 & -4.91 & 8.10 & -3.27 & 3.37 & -4.20 & 5.76 \\
& 0.95 & -2.27 & 1.29 & -1.98 & 0.69 & -2.19 & 1.03 \\
\hline
\multirow{4}{*}{Laplace} 
& 0.70 & -9.34 & 18.64 & -3.57 & 3.27 & -5.93 & 8.90 \\
& 0.80 & -4.69 & 6.46 & -2.73 & 1.46 & -3.81 & 3.89 \\
& 0.95 & -1.76 & 0.19 & -1.46 & 0.04 & -1.68 & 0.13 \\
\hline
\multirow{4}{*}{Logistic} 
& 0.70 & -5.11 & 7.63 & -2.58 & 1.17 & -3.79 & 3.83 \\
& 0.80 & -3.02 & 2.18 & -2.02 & 0.37 & -2.62 & 1.23 \\
& 0.95 & -1.36 & 0.02 & -1.19 & 0.00 & -1.32 & 0.01 \\
\hline
\multirow{3}{*}{Cauchy} 
& 0.70 & -6.21 & 21.30 & -1.76 & 7.26 & -3.87 & 14.74 \\
& 0.80 & -3.13 & 12.54 & -1.09 & 4.62 & -2.24 & 9.14 \\
& 0.95 & -0.66 & 2.83  & -0.25 & 1.09 & -0.50 & 2.14 \\
\hline
\multirow{4}{*}{Known $(\mu,\Gamma,J)_S$} 
& 0.70 & -3.26 & 2.74 & -2.30 & 0.71 & -3.02 & 2.08 \\
& 0.80 & -2.63 & 1.38 & -1.83 & 0.21 & -2.43 & 0.92 \\
& 0.95 & -1.26 & 0.01 & -0.89 & 0.00 & -1.19 & 0.00 \\
\hline
\multirow{4}{*}{Known $(\mu,\Gamma,J)$} 
& 0.70 & -2.73 & 1.58 & -1.72 & 0.13 & -2.44 & 0.93 \\
& 0.80 & -2.04 & 0.45 & -1.34 & 0.02 & -1.87 & 0.23 \\
& 0.95 & -0.91 & 0.00 & -0.63 & 0.00 & -0.85 & 0.00 \\
\hline
\multirow{4}{*}{Bounded norm} 
& 0.70 & -2.68 & 0.00 & -2.33 & 0.00 & -2.56 & 0.00 \\
& 0.80 & -2.30 & 0.00 & -2.08 & 0.00 & -2.23 & 0.00 \\
& 0.95 & -1.65 & 0.00 & -1.59 & 0.00 & -1.64 & 0.00 \\
\hline
\multirow{4}{*}{Data-driven} 
& 0.70 & -0.76 & 0.00 & -0.52 & 0.00 & -0.71 & 0.00 \\
& 0.80 & -0.59 & 0.00 & -0.40 & 0.00 & -0.55 & 0.00 \\
& 0.95 & -0.29 & 0.00 & -0.20 & 0.00 & -0.27 & 0.00 \\
\hline
\end{tabular}
\label{tab:3CP_results}
\end{table}
In Figure~\ref{fig:gap}, we compare the upper and lower SOCP approximations using varying numbers of tangent points, 
with parameters fixed at $p=0.95$, $\theta=2$, $n=10$ decision variables, and $m=5$ chance constraints. As shown, the relative gap $(\mathrm{UB}-\mathrm{LB})/|\mathrm{UB}|$ decreases steadily and approaches zero as the number of tangent points increases, demonstrating the convergence and tightness of the SOCP approximation across all cases.
\begin{figure}[t]
    \centering
    \includegraphics[width=0.5\linewidth]{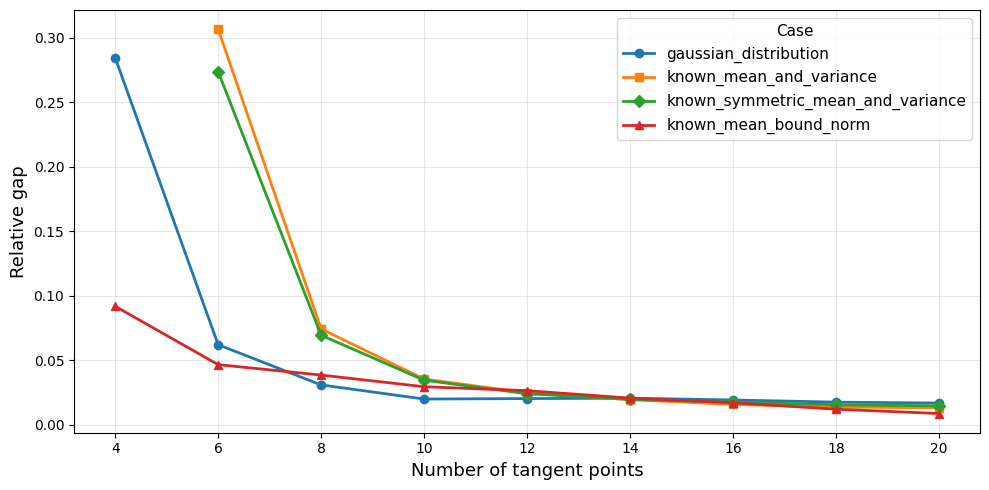}
    \caption{Relative gap $(\mathrm{UB}-\mathrm{LB})/|\mathrm{UB}|$ between the upper and lower SOCP approximations 
    for different numbers of tangent points ($p=0.95$, $\theta=2$, $n=10$, $m=5$).}
    \label{fig:gap}
\end{figure}
In Figure~\ref{fig:violation}, we evaluate the out-of-sample for both the individual and joint formulations under the same setting ($n=50$, $m=15$, $p=0.75$, $\theta=2$). 
For all distributional cases, the observed mean violation remains below the target level $1-p$, verifying that the probabilistic constraints are satisfied in practice.
\begin{figure}[t]
    \centering
    \includegraphics[width=\linewidth]{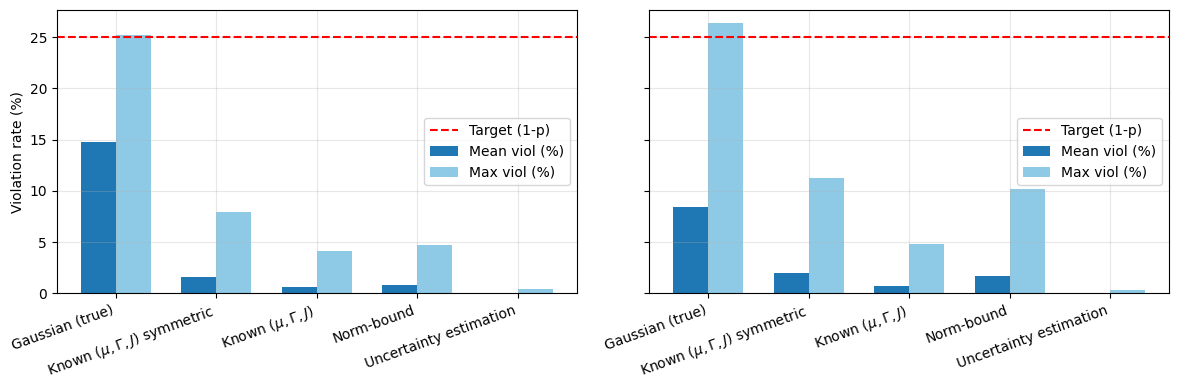}
    \caption{Mean violation probability for the individual (left) and joint (right) formulations with $n=10$, $m=10$, $p=0.75$, and $\theta=4$. The dashed horizontal line indicates the target level $(1-p)$.}
    \label{fig:violation}
\end{figure}
In Figure~\ref{fig:Estimation_Unc}, we investigate the impact of  {estimation uncertainty}. 
We first generate the true parameters $(\mu_d, \Gamma_d, J_d)$ and solve the problem assuming they are exactly known (red line). 
We then generate $N$ random samples from this distribution and re-solve the problem using the empirical estimates 
$(\widehat{\mu}_d, \widehat{\Gamma}_d, \widehat{J}_d)$ for different values of $N$. 
As the number of samples increases, the resulting objective values gradually converge to those obtained under the true parameters, confirming that the empirical estimates converge to the true distribution as $N$ grows.
\begin{figure}[t]
    \centering
    \includegraphics[width=\linewidth]{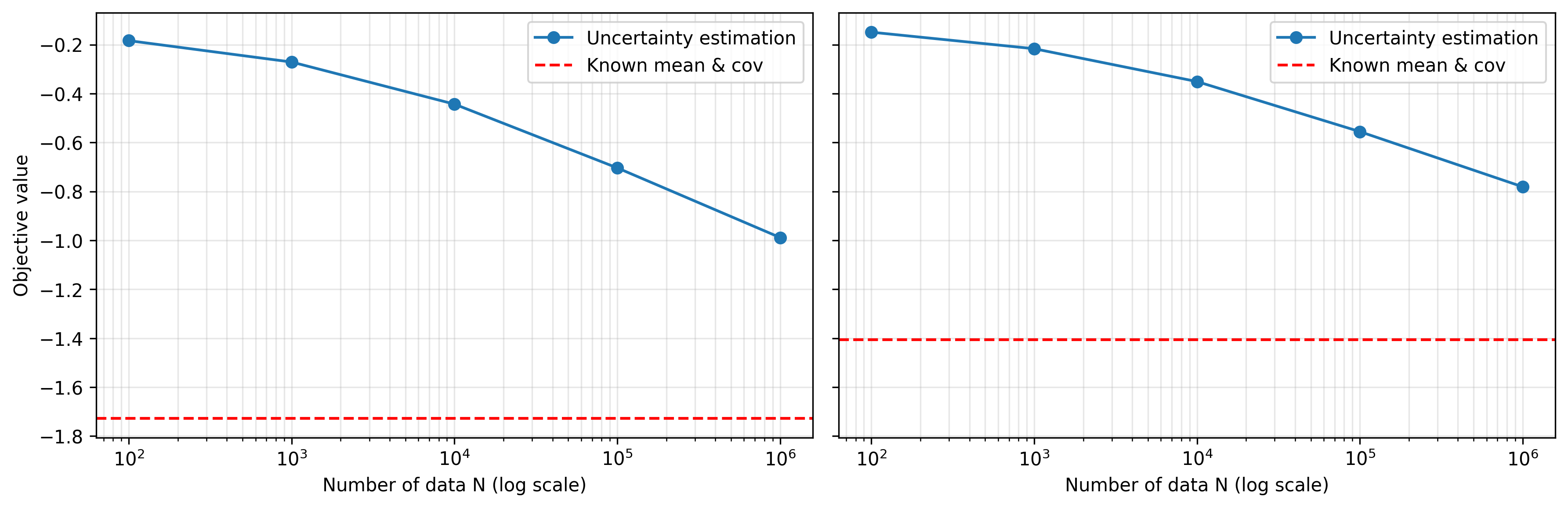}
    \caption{Effect of estimation uncertainty for the individual (left) and joint (right) formulations ($n=10$, $m=10$, $p=0.75$, $\theta=4$). The red line corresponds to the solution obtained with the true $(\mu_d,\Gamma_d,J_d)$.}
    \label{fig:Estimation_Unc}
\end{figure}

\section{Conclusion} \label{sec: 8}
This work developed a unified and tractable framework for the Complex Chance-Constrained Problem (3CP), extending classical chance-constrained formulations to the complex-valued setting. We established deterministic SOCP reformulations for individual chance constraints under general Complex Elliptically Symmetric (CES) distributions, moment-based distributional robustness, bounded-norm uncertainty, and fully data-driven scenarios where only empirical estimates of the mean and covariance are available. To address joint chance constraints, we incorporated dependence via the Gumbel-Hougaard copula and proposed lower and upper SOCP approximations based on first-order Taylor expansions and piecewise linear interpolation in the complex space. Numerical experiments confirmed the theoretical guarantees: for all distributional settings, the empirical violation rates remained below the target level $1-p$, and the SOCP upper-lower gap decreased steadily with more tangent points. The data-driven formulation was shown to be stable, with solutions converging to those obtained using the true parameters as the sample size grows. The application to MVDR beamforming demonstrated that the proposed 3CP framework can be directly applied to practical signal processing problems and remains robust to model mismatch. The complex formulation reduces to the classical real-valued chance-constrained problem as a special case. Thus, all the deterministic SOCP reformulations derived for the complex setting naturally contain the real case as a direct special instance.

\bibliographystyle{plainnat}
\bibliography{bib}

\end{document}